\documentclass[a4paper,reqno]{amsart} 
\usepackage[utf8]{inputenc}
\usepackage{upref}
\usepackage{mathrsfs} 
\usepackage{amsmath,amssymb,amsthm}
\usepackage{mathtools}
\usepackage{thmtools}
\usepackage{enumerate} 
\usepackage{indentfirst} 

\usepackage{hyperref}%
\hypersetup{
hypertexnames=false, 
linkbordercolor=[rgb]{0.5, 1, 1},%
urlbordercolor=[rgb]{0.5, 1, 1},%
citebordercolor=[rgb]{0.5, 1, 0.5}}%
\usepackage{graphicx} 
\usepackage{subcaption} 

\usepackage{cite} 
\usepackage{listings} 
\usepackage{xcolor} 

\definecolor{codegreen}{rgb}{0,0.6,0} 
\definecolor{codegray}{rgb}{0.5,0.5,0.5} 
\definecolor{codepurple}{rgb}{0.58,0,0.82} 
\definecolor{backcolour}{rgb}{0.95,0.95,0.92} 

\lstdefinestyle{mystyle}{
language=Python,
backgroundcolor=\color{backcolour},
commentstyle=\color{codegreen},
keywordstyle=\color{magenta},
numberstyle=\tiny\color{codegray},
stringstyle=\color{codepurple},
basicstyle=\ttfamily\footnotesize,
breakatwhitespace=false,
breaklines=true,
captionpos=b,
keepspaces=true,
numbers=left,
numbersep=5pt,
showspaces=false,
showstringspaces=false,
showtabs=false,
tabsize=2
}
\declaretheoremstyle[
headfont=\normalfont\bfseries,
bodyfont=\normalfont\itshape,
notefont=\mdseries,
notebraces={(}{)},
headindent=\parindent,
postheadspace=0.5em,
headpunct=.,
]{plain}

\declaretheoremstyle[
headfont=\normalfont\itshape,
bodyfont=\normalfont\rmfamily,
notefont=\mdseries,
notebraces={(}{)},
headindent=\parindent,
postheadspace=0.5em,
headpunct=.,
qed=\qedsymbol
]{remark}

\declaretheoremstyle[
headfont=\normalfont\bfseries,
bodyfont=\normalfont\rmfamily,
notefont=\mdseries,
notebraces={(}{)},
headindent=\parindent,
postheadspace=0.5em,
headpunct=.,
]{definition} 

\numberwithin{equation}{section}     

\declaretheorem[numberwithin=section, name=Theorem]{theorem} 
\declaretheorem[sibling=theorem, name=Lemma]{lemma} 
 
\declaretheorem[sibling=theorem, name=Corollary]{corollary*}
\declaretheorem[sibling=theorem, style=remark, name=Remark]{remark} 
\declaretheorem[sibling=theorem, style=remark, name=Example]{example} 

\NewDocumentCommand\setA{}{\mathscr{A}} 
\DeclareMathOperator{\conv}{conv}
\DeclareMathOperator{\absco}{absco} 
\NewDocumentCommand\ldotsb{}{\ldots\,} 

\begin{document} 

\title[Notes on Simplifying the Construction of Barabanov Norms]{Notes on 
Simplifying the Construction of Barabanov Norms}

\author{Victor Kozyakin}

\thanks{The research is supported by the MSHE ``Priority 2030'' strategic academic leadership
program.}

\address{Higher School of Modern Mathematics MIPT,
9 Institutskiy per., Dolgoprudny, Moscow Region, 141701, Russian Federation} 

\email{koziakin.vs@mipt.ru} 

\keywords{Joint/generalized spectral radius, Barabanov norm, 
Dranishnikov-Konyagin bodies, JSR Toolbox} 

\subjclass[2020]{Primary 15A18; Secondary 15A60, 65F15} 

\date{} 

\begin{abstract} 
  To answer the question about the growth rate of matrix products, the 
  concepts of joint and generalized spectral radius were introduced in the 
  1960s. A common tool for finding the joint/generalized spectral radius is 
  the so-called extremal norms and, in particular, the Barabanov norm. The 
  goal of this paper is to try to combine the advantages of different 
  approaches based on the concept of extremality in order to obtain results 
  that are simpler for everyday use. It is shown how the 
  Dranishnikov--Konyagin theorem on the existence of a special invariant body 
  for a set of matrices can be used to construct a Barabanov norm. A modified 
  max-relaxation algorithm for constructing Barabanov norms, which follows 
  from this theorem, is described. Additional techniques are also described 
  that simplify the construction of Barabanov norms under the assumption that 
  some extremal norm is initially known. 
\end{abstract} 

\maketitle 
\section{Introduction}\label{S:intro} 
In various fields of mathematics, control theory, physics, 
etc.~\cite{Jungers:09, Koz:IITP13} the question arises about the growth/decay 
rate of matrix (operator) products with factors from some sets of matrices 
(linear operators)~$\setA$. If the set~$\setA$ consists of one element, this 
question is solved by calculating the spectral radius of the corresponding 
matrix. But in the case when the set~$\setA$ contains more than one element, 
this question turns out to be very complex and does not have any 
algorithmically or computationally ``simple'' answer~\cite{Koz:AiT90:6:e, 
TB:MCSS97:1, Koz:AiT03:9:e}. 

To answer the question about the growth rate of matrix products, in the 
1960s, the analytical concepts of joint~\cite{RotaStr:IM60} and 
generalized~\cite{DaubLag:LAA92} spectral radius of a set of matrices~$\setA$ 
were introduced. Let us recall the corresponding concepts, following the 
works~\cite{Koz:CDC05:e, Koz:INFOPROC05:e, Koz:INFOPROC06:e}. 

Let $\setA=\{A_{1},\ldots,A_{m}\}$ be a set of real $d\times d$ matrices and 
${\|\cdot\|}$ be some norm in~$\mathbb{R}^{d}$. With each finite set of 
symbols $\boldsymbol{\sigma}= 
\{\sigma_{1},\sigma_{2},\ldots,\sigma_{n}\}\in{\{1,\ldots,m\}}^{n}$, where 
${n\ge1}$, we associate the matrix 
$A_{\boldsymbol{\sigma}}=A_{\sigma_{n}}\cdots A_{\sigma_{2}}A_{\sigma_{1}}$ 
and define two numerical quantities: 
\begin{equation}\label{eq:defrA} 
  \rho_{n}({\setA})=\max_{\boldsymbol{\sigma}\in{\{1,\ldots,m\}}^{n}} 
  \|A_{\boldsymbol{\sigma}}\|^{1/n},\qquad 
  \bar{\rho}_{n}({\setA})=\max_{\boldsymbol{\sigma}\in{\{1,\ldots,m\}}^{n}} 
  {\rho(A_{\boldsymbol{\sigma}})}^{1/n}, 
\end{equation} 
where $\rho(\cdot)$ denotes the spectral radius of a matrix. In this 
notation, the limit 
\begin{equation}\label{eq:JSR} 
  \rho({\setA})= \limsup_{n\to\infty}\rho_{n}({\setA}), 
\end{equation} 
which does not depend on the choice of the norm $\|\cdot\|$ and in fact 
coincides with the limit $\rho({\setA})=\lim_{n\to\infty}\rho_{n}({\setA})$, 
is called the \emph{joint spectral radius} of the set of matrices 
$\setA$~\cite{RotaStr:IM60}. Similarly, we can consider the limit 
\begin{equation}\label{eq:GSR} 
  \bar{\rho}({\setA})= \limsup_{n\to\infty}\bar{\rho}_{n}({\setA}), 
\end{equation} 
called the \emph{generalized spectral radius} of the matrix 
set~$\setA$~\cite{DaubLag:LAA92}. For bounded matrix sets~$\setA$, the 
quantities $\rho({\setA})$ and $\bar{\rho}({\setA})$ coincide with each 
other~\cite{BerWang:LAA92}, and 
\begin{equation}\label{eq:sprad} 
  \bar{\rho}_{n}({\setA})\le \bar{\rho}({\setA})=\rho({\setA})\le 
  \rho_{n}({\setA}),\qquad\forall~n. 
\end{equation} 

Finding the values of $\rho({\setA})$ and $\bar{\rho}({\setA})$ turned out to 
be a rather difficult task in theoretical and algorithmic/computational 
terms. Nevertheless, quite meaningful and strong approaches were developed 
along this path. In particular, several algorithms were developed for 
calculating the values of $\rho({\setA})$ and 
$\bar{\rho}({\setA})$~\cite{GugZen:CDC05, Koz:DCDSB10, GugProt:FCM13, 
VHJ:ACM14}. At present, the most developed set of algorithms is probably the 
one described in~\cite{Mejstrik:ACMTMS20, MejReif:LAA25}. This set of 
algorithms combines many earlier algorithms from other authors and is 
implemented as an extension package 
\href{https://gitlab.com/tommsch/ttoolboxes}{t-toolboxs} for the MATLAB 
program. 

In the late 1980s, in~\cite{Bar:AIT88-2:e,Bar:AIT88-3:e,Bar:AIT88-5:e}, a 
geometric approach to the problem of estimating the growth rate of matrix 
products was proposed, which subsequently became one of the main methods of 
analysis in this theory. This approach consists of proving the existence, for 
a set of matrices~$\setA$, of norms or invariant sets satisfying certain 
special relations. Among the subsequent works, we 
highlight~\cite{Wirth:LAA02}. 

Let us recall the relevant facts. If for some norm $\|\cdot\|$ a number 
$\rho>0$ can be found such that the identity 
\begin{equation}\label{eq:mane-bar} 
  \max_{i}\|A_{i}x\|\equiv\rho(\setA)\|x\| 
\end{equation} 
is satisfied, then such a norm is called the \emph{Barabanov norm}, 
corresponding to the set of matrices~$\setA$. For brevity, such a norm will 
be further called the \emph{\textit{B}-norm}. 

\begin{theorem}[N.E. Barabanov]\label{T:Bar} 
  Let the matrix set $\setA=\{A_{1},\ldots,A_{m}\}$ be irreducible\footnote{A 
  matrix set~$\setA$ is called \emph{irreducible} if the matrices 
  from~$\setA$ do not have common invariant subspaces distinct from $\{0\}$ 
  and ${\mathbb{R}}^{d}$.}. Then the number~$\rho$ is a joint 
  $($generalized$)$ spectral radius of~$\setA$ if and only if there exists a 
  norm $\|\cdot\|$ in ${\mathbb{R}}^{d}$ satisfying 
  identity~\eqref{eq:mane-bar}. 
\end{theorem} 

The proof of this theorem is simple~\cite[Thm.~2]{Bar:AIT88-2:e}, but 
unfortunately the corresponding norm $\|\cdot\|$ is defined in it as the 
result of some computationally non-constructive limit procedure. 

Another approach to finding the values of $\rho({\setA})$ and 
$\bar{\rho}({\setA})$ was developed in~\cite{PWB:CDC05, Prot:FPM96:e, 
PW:LAA08}. It is based on the following statement. 

\begin{theorem}[A.N. Dranishnikov, S.V. Konyagin]\label{T:DK} 
  If the set of matrices $\setA=\{A_{1},\ldots,A_{m}\}$ is irreducible, then 
  there exists a convex body\footnote{A body is a set with nonempty 
  interior.}~$M$ such that 
  \begin{equation}\label{eq:DKbody} 
    \rho M = \conv \left(\bigcup_{i}A_{i}M\right) 
  \end{equation} 
  for some $\rho>0$. Moreover, for any centrally symmetric body~$M$ 
  satisfying~\eqref{eq:DKbody} for some $\rho$, the equality $\rho = 
  \rho(\setA)$ holds. 
\end{theorem} 

A body~$M$ satisfying~\eqref{eq:DKbody} for some $\rho$ is called a 
\emph{Dranishnikov-Konyagin body} (hereinafter the name 
\emph{\textit{DK}-body} will be also used). One of the first complete proofs 
of Theorem~\ref{T:DK} was given in~\cite{Prot:FPM96:e}. An algorithm for 
constructing \textit{DK}-bodies was also proposed there, but it was not 
widely used in the literature. 

Each centrally symmetric body~$M$ can be treated as a unit ball of some norm, 
called the Minkowski norm of the corresponding body. For a 
\textit{DK}-body~$M$ this norm is defined by the equality 
\begin{equation}\label{eq:dknorm} 
  \|x\|_{dk}:=\min\left\{t: t\ge0,~ x\in tM\right\}. 
\end{equation} 
The norm $\|\cdot\|_{dk}$ was first, apparently, studied in detail 
in~\cite{Prot:FPM96:e}. Since then it has been called the 
\emph{Dranishnikov--Konyagin--Protasov norm} (hereinafter referred to as the 
\emph{\textit{DKP}-norm}). 

Each \textit{DKP}-norm, as well as each \textit{B}-norm, is a so-called 
\emph{extremal norm} corresponding to a set of matrices $\setA$, i.e.\ a norm 
satisfying 
\begin{equation}\label{eq:extnorm} 
  \|A_{i}x\|\le\rho(\setA)\|x\|, 
  \qquad\forall~A_{i}\in\setA,~\forall~x\in\mathbb{R}^{d} 
\end{equation} 
or, equivalently, 
\[ 
  \max_{i}\|A_{i}x\|\le \rho(\setA)\|x\|, \qquad\forall~x\in\mathbb{R}^{d}. 
\] 
The term ``extremal norm'' first appeared, apparently, in~\cite{Bar:ACC95}; a 
list of further works devoted to the study of extremal norms can be found 
in~\cite{Koz:IITP13}. 

With such a powerful tool as the 
\href{https://gitlab.com/tommsch/ttoolboxes}{t-toolboxs} extension package 
for MATLAB described in~\cite{Mejstrik:ACMTMS20, MejReif:LAA25}, further 
attempts to develop algorithms for finding the generalized/joint spectral 
radius would seem pointless. Note, however, that the practical application of 
the \texttt{t-toolboxs} package has a number of limitations: firstly, this 
package is mainly focused on calculating the generalized/joint spectral 
radius, and constructing extremal norms with it, in particular 
\textit{B}-norms, requires some additional effort; and secondly, this package 
is quite large ($\sim15$ Mb) and is intended for use in the (commercial=paid) 
MATLAB program. In this regard, the issue of developing a geometric approach 
(preferably simple in algorithmic terms) to the problem of estimating the 
growth rate of matrix products based on finding the \textit{B}-norm remains 
relevant. 

The aim of this paper is to try to combine the advantages of different 
approaches based on the concept of extremality in order to obtain results 
that are easier to use in everyday life. The structure of the paper is as 
follows: in the introduction (Section~\ref{S:intro}) we tried to justify the 
aim of this paper. The central part of the paper is occupied by 
Section~\ref{S:tBDK}, which establishes the equivalence of the Barabanov and 
Dranishnikov--Konyagin theorems and shows how the Dranishnikov--Konyagin 
theorem can be used to construct Barabanov norms. Section~\ref{S:DK} 
describes the algorithm for constructing Dranishnikov--Konyagin bodies and 
presents the idea of its proof. Finally, Section~\ref{S:extnorms} describes 
additional techniques that allow us to simplify the construction of Barabanov 
norms under the assumption that we initially know some extremal norm. 

\section{Theorems of Barabanov and Dranishnikov--Konyagin}\label{S:tBDK} 

Extremal norms, and even more so \textit{B}-norms, can be found explicitly 
only in rare cases. Nevertheless, such norms exist under fairly general 
assumptions! For example, Theorem~\ref{T:Bar} implies that for irreducible 
matrix sets a \textit{B}-norm always exists, since for such matrix sets there 
always exist limits~\eqref{eq:JSR} and~\eqref{eq:GSR}, which define the joint 
and generalized spectral radius, respectively. And Theorem~\ref{T:DK} implies 
the existence of a \textit{DK}-body, which can be treated as a unit ball of 
the extremal norm~\eqref{eq:dknorm}. 

\begin{remark}\label{rem:maxrate} 
  In cases where there is an extremal norm, for each $x\in\mathbb{R}^{d}$ and 
  each finite sequence of indices $\boldsymbol{\sigma}= 
  \{\sigma_{1},\sigma_{2},\ldots,\sigma_{n}\}\in{\{1,\ldots,m\}}^{n}$, 
  by~\eqref{eq:extnorm}, the following inequalities hold: 
  \[ 
    \|A_{\sigma_{n}}\cdots A_{\sigma_{2}}A_{\sigma_{1}}x\|\le 
    {\rho({\setA})}^{n}\|x\|\quad\Longrightarrow\quad \|A_{\sigma_{n}}\cdots 
    A_{\sigma_{2}}A_{\sigma_{1}}\|\le {\rho({\setA})}^{n}, 
  \] 
  giving an estimate of the maximum growth rate of matrix products with 
  factors from~$\setA$. 
\end{remark} 
\begin{remark}\label{rem:maxrate-bar} 
  In those cases where the conditions of Theorem~\ref{T:Bar} are satisfied, 
  i.e.\ there exists a \textit{B}-norm $\|\cdot\|$, for each 
  $x\in\mathbb{R}^{d}$ there exists an infinite sequence of indices 
  $\boldsymbol{\sigma}= \{\sigma_{1},\sigma_{2},\ldots\}$ such that for each 
  $n$ in the \textit{B}-norm $\|\cdot\|$ the equality 
  \[ 
    \|A_{\sigma_{n}}\cdots A_{\sigma_{2}}A_{\sigma_{1}}x\|= 
    {\rho({\setA})}^{n}\|x\| 
  \] 
  holds which allows one to ``explicitly'' construct a sequence of matrices 
  with the maximum growth rate in the \textit{B}-norm $\|\cdot\|$. Moreover, 
  relation~\eqref{eq:mane-bar} allows us to construct for any initial 
  condition $x_{0}$ individual (fastest growing) trajectories of the form 
  \[ 
    x_{n+1}=A_{\sigma_{n}}x_{n},\qquad 
    \sigma_{n}\in\{1,2,\ldots,d\},~n=0,1,\ldots\,, 
  \] 
  satisfying the relation 
  \[ 
    \|x_{n+1}\|=\rho(\setA)\|x_{n}\|\quad\Longrightarrow\quad \|x_{n}\|= 
    \rho(\setA)^{n}\|x_{0}\|. 
  \] 
\end{remark} 

As noted above, extremal norms and, in particular, \textit{B}-norms or 
\textit{DKP}-norms are difficult to find explicitly, even despite the 
explicit form of the expressions defining them~\eqref{eq:mane-bar} 
and~\eqref{eq:DKbody}. In addition, although the concepts of all these norms 
are close to each other, they are still different. At the same time, the 
concept of the Barabanov norm is more informative in terms of applications 
(see Remark~\ref{rem:maxrate-bar}), but is also more difficult to verify, 
while the concept of an extremal norm, although less informative, is easier 
to verify. In this regard, at least the following questions arise: 
\begin{enumerate} 
\item If some extremal norm is known, does this somehow simplify the search 
    for \textit{B}-norms or \textit{DKP}-norms? 
\item Can knowledge of a \textit{DKP}-norm make it easier to construct a 
    \textit{B}-norm, and vice versa? 
\item Are there any situations in which extremal norms can be found in some 
    reasonable sense ``simply''? 
\end{enumerate} 

The remainder of this section and Section~\ref{S:DK} are devoted to partial 
answers to the first two questions. Some situations related to the third 
question will be considered in Section~\ref{S:extnorms}. Finally, 
Appendix~\ref{A:PyCode} contains modified Python code for finding 
\textit{B}-norms and \textit{DKP}-norms. 

In~\cite{PW:LAA08} the proof of Theorem~\ref{T:DK} used the technique of 
\emph{dual norms}, see, for example,~\cite[Sect.~5.4, 5.5]{HJ2:e}. It will be 
more convenient for us to use another, conceptually close, but technically 
somewhat different derivation of the proof of Theorem~\ref{T:DK} from 
Theorem~\ref{T:Bar}, using the technique of polars from the theory of duality 
of vector spaces, see, for example,~\cite[Ch.~II]{RobRob:e}. 

\subsection{An equivalent formulation of Barabanov's theorem} 
Denote by $S=\{x:\|x\|\le1\}$ the unit ball in the \textit{B}-norm 
$\|\cdot\|$. Then, in the case where all matrices $A_{i}\in\setA$ are 
non-singular, identity~\eqref{eq:mane-bar} can be rewritten in an equivalent 
form in terms of the set $S$: 
\begin{equation}\label{eq:BarS} 
  S=\rho\,\bigcap_{i}A_{i}^{-1}S. 
\end{equation} 
Therefore, in this case, Theorem~\ref{T:Bar} admits an equivalent 
formulation: 
\begin{theorem}\label{T:Bar-alt} 
  Let the set of non-singular matrices $\setA=\{A_{1},\ldots,A_{m}\}$ be 
  irreducible. Then the number~$\rho$ is a joint $($generalized$)$ spectral 
  radius of~$\setA$ if and only if there exists a centrally symmetric convex 
  body $S$ satisfying equality~\eqref{eq:BarS}. 
\end{theorem} 

The given reformulation of the Barabanov Theorem~\ref{T:Bar} is close to the 
formulation of the Dranishnikov--Konyagin Theorem~\ref{T:DK}. At the same 
time, it turns out to be more suitable for software implementation of 
algorithms for finding the \textit{B}-norm. For a more detailed comment, see 
the following Remark~\ref{rem:other2}. 

\subsection{Equivalence of the Barabanov and Dranishnikov--Konyagin 
theorems}\label{SS:BeqDK} Let us recall the necessary concepts, restricting 
ourselves to the case of finite-dimensional spaces. A \emph{polar} of a set 
$X\subset\mathbb{R}^{d}$ is the set $X^{\circ}$ of all $x'\in\mathbb{R}^{d}$ 
for which 
\[ 
  \sup\{|\langle x,x'\rangle|: x\in X\}\le 1, 
\] 
where $\langle x,x'\rangle$ denotes the bilinear form 
\[ 
  \langle x,x'\rangle=x_{1}x'_{1}+\cdots+x_{m}x'_{m},\qquad 
  x,x'\in\mathbb{R}^{d}. 
\] 

Polars of sets from $\mathbb{R}^{d}$ have the following 
properties~\cite{RobRob:e}: 
\begin{enumerate}[(i)]
\item the set $X^{\circ}$ is absolutely convex\footnote{The
    \emph{absolutely convex hull} of a set $X$ is the absolutely convex 
    symmetric closure of the set $X$, that is, the set $\absco X:= \{tx+sy: 
    \forall~x,y\in X,~|t|+|s|\le1\}:= \conv\{\{0\}\cup X\cup (-X)\}$. The 
    set $X$ is called absolutely convex if it coincides with its absolutely 
    convex hull: $X=\absco X$.} and closed; 
\item $X^{\circ}=(\absco X)^{\circ}$ and
    $X^{\circ\circ}:=(X^{\circ})^{\circ} =\absco X$; 
\item if $X\subseteq Y$, then $Y^{\circ}\subseteq X^{\circ}$;
\item if $\lambda\neq0$, then $(\lambda X)^{\circ}=\lambda^{-1}X^{\circ}$; 
\item if $A$ is a non-degenerate linear mapping from $\mathbb{R}^{d}$ to 
    $\mathbb{R}^{d}$ ($d\times d$ matrix), then $(A 
    X)^{\circ}=(A^{T})^{-1}X^{\circ}$; 
\item if $\{X_{i}\}$ is a finite collection of sets, then
    $\left(\bigcup_{i} X_{i}\right)^{\circ}=\bigcap_{i}X_{i}^{\circ}$; 
\item if $\{X_{i}\}$ is a finite collection of sets, then
    $\left(\bigcap_{i}X_{i}\right)^{\circ}=\absco\left(\bigcup_{i}X_{i}^{\circ}\right)$. 
\end{enumerate} 

Now we can show that the Barabanov and Dranishnikov-Konyagin theorems are in 
a certain sense equivalent to each other. 

Let the conditions of Theorem~\ref{T:Bar} be satisfied for an irreducible set 
of matrices $\setA=\{A_{1},\ldots,A_{m}\}$. In this case, the number $\rho$ 
is the joint and, hence, the generalized spectral radius of the set of 
matrices~$\setA$. Then, the same number $\rho$ is the generalized and, hence, 
the joint spectral radius of the (irreducible) set of matrices 
$\setA^{T}=\{A^{T}_{1},\ldots,A^{T}_{m}\}$. In this case, by 
Theorem~\ref{T:Bar}, there exists a norm~$\|\cdot\|$ in~${\mathbb{R}}^{d}$ in 
which identity~\eqref{eq:mane-bar} holds. 

Now let us set $M=S^{\circ}$. Then, taking the polars of both parts of 
equality~\eqref{eq:BarS}, we obtain: 
\begin{align*} 
  M=S^{\circ}&= 
  \left(\rho\bigcap_{i}(A^{T}_{i})^{-1}S\right)^{\circ}\stackrel{(1)}{=} 
  \rho^{-1}\left(\bigcap_{i}(A^{T}_{i})^{-1}S\right)^{\circ}
  \\&\stackrel{(2)}{=}\rho^{-1}\absco\left(\bigcup_{i}
    \left((A^{T}_{i})^{-1}S\right)^{\circ}\right)\stackrel{(3)}{=} 
  \rho^{-1}\absco\left(\bigcup_{i}A_{i}S^{\circ}\right)
  \\&=\rho^{-1}\absco\left(\bigcup_{i}A_{i}M\right). 
\end{align*} 
Here equalities (1), (2) and (3) follow from properties (iv), (vii) and (v) 
of polars, respectively, and the remaining equalities follow from the 
definition of the set~$M$. Since polars of any sets are centrally symmetric 
bodies, it follows from the obtained equalities that \[ 
M=\rho^{-1}\absco\left(\bigcup_{i}A_{i}M\right)= 
\rho^{-1}\conv\left(\bigcup_{i}A_{i}M\right), 
\] 
and therefore~$M$ is a \textit{DK}-body for the set of matrices $\setA$. 

Similar calculations show that Theorem~\ref{T:DK} implies 
Theorem~\ref{T:Bar}. 

\begin{remark}\label{rem:alert} 
  In the above reasoning, the non-degeneracy of the matrices $A_{i}$ was 
  implicitly assumed, since in a number of places the matrices $A^{-1}_{i}$ 
  and $(A^{T})^{-1}_{i}$ appeared. We omit the (somewhat more cumbersome) 
  calculations showing that in this case the proposed scheme of reasoning 
  remains valid. 
\end{remark} 

\section{Construction of Dranishnikov--Konyagin bodies}\label{S:DK} 
As in the case of Barabanov's theorem, Dranishnikov--Konyagin's theorem does 
not provide any constructive information on how to find the corresponding 
\textit{DK}-body~$M$. At the same time, there are a number of algorithms for 
computing \textit{B}-norms. In this section, we use the idea of the 
max-relaxation algorithm from~\cite{Koz:DCDSB10, Koz:ArXiv10:1}, presented in 
the next subsection, to iteratively construct the \textit{DK}-body~$M$. 

\subsection{Max-relaxation algorithm for constructing Barabanov 
norms}\label{S:alrMR} A continuous function $\gamma(t,s)$, $t,s> 0$, with the 
properties
\[
\gamma(t,t)=t,\qquad 
\min\{t,s\}<\gamma(t,s)<\max\{t,s\}\quad\text{for}\quad t\neq s, 
\]
will be called an averaging function in what follows. Examples of averaging 
functions are the functions $\gamma(t,s)=\frac{t+s}{2}$, 
$\gamma(t,s)=\sqrt{ts}$, $\gamma(t,s)=\frac{2ts}{t+s}$. 

Let $\|\cdot\|_{0}$ and an arbitrary element $e\neq0$ such that $\|e\|_{0}=1$ 
be given in ${\mathbb{K}}^{m}$, and let $\gamma(\cdot,\cdot)$ be an averaging 
function. We construct recursively a sequence of norms $\|\cdot\|_{n}$ 
according to the following rules: 

MR$_{1}$: assuming that the norm $\|\cdot\|_{n}$ is already known, we 
calculate the values 
\begin{equation}\label{eq:rho} 
  \rho^{+}_{n}=\max_{x\neq0}\frac{\max\limits_{i}\|A_{i}x\|_{n}}{\|x\|_{n}},\quad 
  \rho^{-}_{n}=\min_{x\neq0}\frac{\max\limits_{i}\|A_{i}x\|_{n}}{\|x\|_{n}}; 
\end{equation} 

MR$_{2}$: we set $\gamma_{n}=\gamma(\rho^{-}_{n},\rho^{+}_{n})$ and define a 
new norm: 
\begin{equation}\label{eq:max-relax} 
  \|x\|_{n+1}= 
  \max\left\{\|x\|_{n},~\gamma^{-1}_{n}\max_{i}\|A_{i}x\|_{n}\right\}, 
\end{equation} 
after which we calibrate the norm $\|\cdot\|_{n+1}$ by setting 
\begin{equation}\label{eq:calibr} 
  \|x\|^{\bullet}_{n+1}=\|x\|_{n+1}/\|e\|_{n+1}. 
\end{equation} 

\begin{theorem}[see~\cite{Koz:DAN09:e, Koz:DCDSB10, 
Koz:ArXiv10:1}]\label{th:1} For any irreducible set of matrices~$\setA$ and 
any averaging function $\gamma(t,s)$, the sequences $\{\rho^{\pm}_{n}\}$ 
defined by the iterative procedure MR$_{1}$, MR$_{2}$ converge to 
$\rho(\setA)$, and the sequence of norms $\|\cdot\|^{\bullet}_{n}$ uniformly 
on each bounded set converges to some \textit{B}-norm $\|\cdot\|^{*}$ of the 
set of matrices~$\setA$. Moreover, the sequence $\{\rho^{-}_{n}\}$ does not 
decrease, and the sequence $\{\rho^{+}_{n}\}$ does not increase and 
$\rho^{-}_{n}\le \rho(\setA)\le \rho^{+}_{n}$ for $n=1,2,\ldotsb$, which 
provides an a posteriori estimate of the error in calculating $\rho(\setA)$. 
\end{theorem} 

\subsection{Construction of Dranishnikov-Konyagin bodies}\label{S:DKalg} 
By analogy with the max-relaxation algorithm for \textit{B}-norms, we propose 
the following algorithm for constructing \textit{DK}-bodies, which is a 
simple generalization of the max-relaxation algorithm from~\cite{Koz:DCDSB10, 
Koz:ArXiv10:1}. The key idea in this algorithm will be the remark 
from~\cite{PWB:CDC05, PW:LAA08} on the duality of \textit{DK}-bodies and unit 
balls of \textit{B}-norms. 

Let some centrally symmetric convex body $M_{0}\subset{\mathbb{K}}^{m}$ and 
an arbitrary element $e\neq0$ be given, and let $\gamma(\cdot,\cdot)$ be an 
averaging function. Let us construct recursively a sequence of centrally 
symmetric bodies $M_{n}$ according to the following rules\footnote{CHR 
hereinafter stands for ``Convex Hull Relaxation''.}: 

CHR$_{1}$: assuming that the body $M_{n}$ is already known, we calculate the 
quantities 
\begin{align*} 
  \rho^{+}_{n} & =\min\left\{\rho: \conv\left(\bigcup_{i}A_{i}M_{n}\right) 
  \subseteq\rho M_{n}\right\}, \\ \rho^{-}_{n} & =\max\left\{\rho: \rho M_{n} 
  \subseteq\conv\left(\bigcup_{i}A_{i}M_{n}\right)\right\}; 
\end{align*} 

CHR$_{2}$: we set $\gamma_{n}=\gamma(\rho^{-}_{n},\rho^{+}_{n})$ and define a 
new centrally symmetric convex body: 
\begin{equation}\label{eq:conv-relax} 
  M_{n+1}= \conv\left\{M_{n},\gamma^{-1}_{n}\bigcup_{i}A_{i}M_{n}\right\}, 
\end{equation} 
after which we calibrate the body $M_{n+1}$ by setting 
\begin{equation}\label{eq:calibrM} 
  M^{\bullet}_{n+1}=\mu_{n+1}M_{n+1}, 
\end{equation} 
where $\mu_{n+1}$ is chosen such that the vector $e$ belongs to the boundary 
of the body~$M^{\bullet}_{n+1}$. 

In this case, the following statement will hold. 

\begin{theorem}\label{th:2} 
  For any irreducible set of matrices~$\setA$ and any averaging function 
  $\gamma(t,s)$, the sequences $\{\rho^{\pm}_{n}\}$ defined by the iterative 
  procedure CHR$_{1}$, CHR$_{2}$ converge to $\rho(\setA)$, and the sequence 
  of centrally symmetric convex bodies $M^{\bullet}_{n}$ converges (in the 
  Hausdorff metric) to some \textit{DK}-body~$M^{*}$ of the set of 
  matrices~$\setA$. Moreover, the sequence $\{\rho^{-}_{n}\}$ does not 
  decrease, and the sequence $\{\rho^{+}_{n}\}$ does not increase and 
  $\rho^{-}_{n}\le \rho(\setA)\le \rho^{+}_{n}$ for $n=1,2,\ldotsb$, which 
  provides an a posteriori estimate of the error in calculating 
  $\rho(\setA)$. 
\end{theorem} 

\begin{example}\label{ex2} 
  Consider a set of matrices $\setA=\{A_{1},A_{2}\}$, where 
  \[ 
    A_{1}=0.576 
    \begin{bmatrix}0.9 & 1.1 \\0 & 1 
    \end{bmatrix},\quad 
    A_{2}=0.8 
    \begin{bmatrix}1 & 0 \\1.0 & 0.9 
    \end{bmatrix}. 
  \] 
  For these matrices $\rho=\rho(\setA)=1.098668$; the unit ball 
  $S=\{x:\|x\|\le1\}$, the \textit{B}-norm $\|\cdot\|$ and the 
  \textit{DK}-body~$M$ are shown in Fig.~\ref{F:1}. In this figure, the 
  \textit{DK}-body~$M$ is denoted by the black solid line, the set $S$ is 
  denoted by the green solid line, the set $\rho^{-1}A_{1}M$ is denoted by 
  the red dotted line, and the set $\rho^{-1}A_{2}M$ is denoted by the blue 
  dash-dotted line. 
\end{example}
  \begin{figure}[htbp!] 
    \centering 
    \includegraphics*[width=0.48\textwidth]{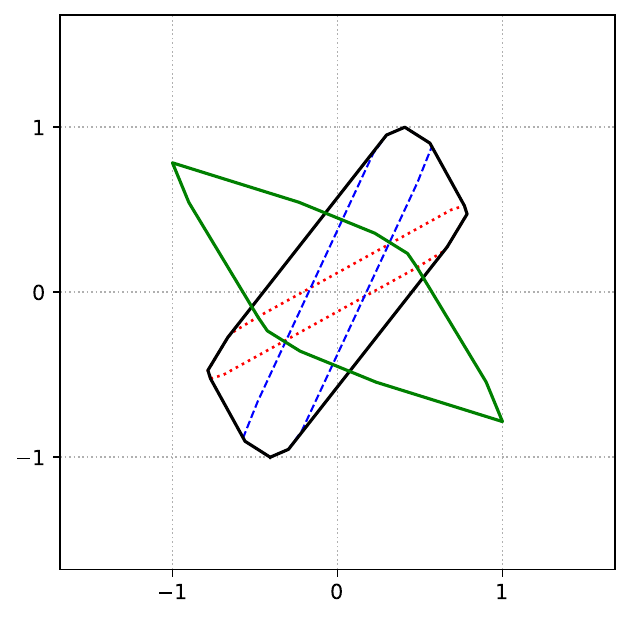} 
    \caption{Unit ball of the Barabanov norm and the Dranishnikov--Konyagin 
    body}\label{F:1} 
  \end{figure}

\subsection{Scheme of proof of Theorem~\ref{th:2}}\label{S:th2-proof} 
To prove Theorem~\ref{th:2}, one could almost verbatim repeat the proof of 
Theorem~\ref{T:Bar} presented in~\cite{Bar:AIT88-2:e}. However, the 
corresponding constructions in~\cite{Bar:AIT88-2:e} are rather cumbersome, 
and repeating them would not bring anything new in terms of ideas. Therefore, 
below we describe the idea of deriving the proof~\ref{th:2} directly from the 
statement of Theorem~\ref{T:Bar}, relying on the polar technique and not 
using the arguments from~\cite{Bar:AIT88-2:e}. In doing so, we use arguments 
close to those used in Section~\ref{S:tBDK} when substantiating the 
equivalence of the Barabanov and Dranishnikov--Konyagin theorems. 

Let $\{M_{n}\}$, $n=0,1,\ldots$, be a set of centrally symmetric convex 
bodies satisfying conditions CHR$_{1}$ and CHR$_{2}$ from 
Subsection~\ref{S:DKalg}. Let us associate each body $M_{n}$ with its polar: 
\[ 
  S_{n}=M_{n}^{\circ},\qquad n=0,1,2,\ldots\,, 
\] 
and take the polars from the left and right parts of 
equality~\eqref{eq:conv-relax}: 
\[ 
  S_{n+1}=M_{n+1}^{\circ}=\left( 
  \conv\left\{M_{n},\gamma^{-1}_{n}\bigcup_{i}A_{i}M_{n}\right\}\right)^{\circ}. 
\] 
Then, just as it was shown in section~\ref{S:tBDK}, we get: 
\begin{align*} 
  S_{n+1}=M_{n+1}^{\circ}&=  \left( 
  \conv\left\{M_{n},\gamma^{-1}_{n}\bigcup_{i}A_{i}M_{n}\right\}\right)^{\circ}
  \\&= 
  \left( 
  \conv\left\{M_{n}\bigcup\left(\gamma^{-1}_{n}\bigcup_{i}A_{i}M_{n}\right)\right\}\right)^{\circ} 
  \\ 
     &= 
     M_{n}^{\circ}\bigcap\left(\gamma^{-1}_{n}\bigcup_{i}A_{i}M_{n}\right)^{\circ}= 
     M_{n}^{\circ}\bigcap\left(\gamma_{n}\left(\bigcup_{i}A_{i}M_{n}\right)^{\circ}\right)
  \\ 
     &= 
     M_{n}^{\circ}\bigcap\left(\gamma_{n}\bigcap_{i}\left(A_{i}M_{n}\right)^{\circ}\right)= 
     M_{n}^{\circ}\bigcap\left(\gamma_{n}\bigcap_{i}(A_{i}^{T})^{-1}M_{n}^{\circ}\right)
  \\ 
     &= 
     S_{n}\bigcap\left(\gamma_{n}\bigcap_{i}(A_{i}^{T})^{-1}S_{n}\right), 
\end{align*} 
whence 
\begin{equation}\label{eq:Siter} 
  S_{n+1}= 
  S_{n}\bigcap\left(\gamma_{n}\bigcap_{i}(A_{i}^{T})^{-1}S_{n}\right), \qquad 
  n=0,1,\ldots\,. 
\end{equation} 

Let us now consider for each body $S_{n}$ its Minkowski norm: 
\[ 
  \|x\|_{n}:=\min\left\{t: t\ge0,~ x\in tS_{n}\right\}. 
\] 
Then, just as it was shown in Section~\ref{S:tBDK}, we obtain: 
\[ 
  \|x\|_{n+1}=\max\left\{\|x\|_{n}, 
  \gamma_{n}^{-1}\max_{i}\|A_{i}^{T}x\|_{n}\right\}, \qquad n=0,1,\ldots\,, 
\] 
i.e.\ for the norms $\|\cdot\|_{n}$ equalities~\eqref{eq:max-relax} of the 
max-relaxation algorithm for the set of matrices 
$\setA^{T}=\{A^{T}_{1},\ldots,A^{T}_{m}\}$ from Subsection~\ref{S:alrMR} will 
be satisfied. In this case, the relations defining the quantities 
$\rho_{n}^{+}$ and $\rho_{n}^{-}$ in terms of the Minkowski 
norms~$\|\cdot\|_{n}$ defined above take exactly the form~\eqref{eq:rho} 
(with the matrices $A_{i}$ replaced by their transposes $A^{T}_{i}$) from the 
MR$_{1}$ condition of the max-relaxation algorithm (see 
Subsection~\ref{S:alrMR}). 

According to the statement of Theorem~\ref{th:1} after 
calibration~\eqref{eq:calibr}, the obtained norms 
\begin{equation}\label{eq:calibr1} 
  \|x\|^{\bullet}_{n}=\|x\|_{n}/\|e\|_{n} 
\end{equation} 
will converge to some \textit{B}-norm $\|\cdot\|^{*}$. Rewriting 
relation~\eqref{eq:calibr1} in terms of balls $S_{n}=M_{n}^{\circ}$, and then 
passing to polars $S_{n}^{\circ}=M_{n}^{\circ\circ}=M_{n}$, we obtain that 
the sequence $\{\|e\|_{n}^{-1}M_{n}\}$ of ``calibrated'' bodies $\{M_{n}\}$ 
will converge to some \textit{DK}-body. As is easy to show, then the sequence 
of calibrated bodies $M^{\bullet}_{n}=\mu_{n}M_{n}$ (see~\eqref{eq:calibrM}), 
where $\mu_{n}$ is chosen so that the vector $e$ belongs to the boundary of 
the body $M^{\bullet}_{n}$, will also converge to some \textit{DK}-body. 

\begin{remark}\label{rem:other1} 
  Let us emphasize that the above scheme of proof of Theorem~\ref{th:2} is 
  really only a scheme of proof and requires a full clarification of some 
  technical details. 
\end{remark} 

\begin{remark}\label{rem:other2} 
  When implementing the max-relaxation algorithm, there is a need to 
  sequentially calculate the norms $\|\cdot\|_{n}$, which is not an obvious 
  task in terms of software implementation, since the numerical specification 
  of a norm, like any other function, may require significant computational 
  resources. To simplify this task and minimize possible calculation errors, 
  the most convenient way turned out to be to consider the so-called 
  polygonal norms $\|\cdot\|_{n}$, i.e.\ norms whose unit balls~$S_{n}$ are 
  convex centrally symmetric polyhedra. With this approach, the algorithm for 
  calculating the norms $\|\cdot\|_{n}$ is actually reduced to sequentially 
  calculating the polyhedra~$S_{n}$ using formula~\eqref{eq:Siter}. In this 
  case, all the necessary geometric transformations are reduced to applying 
  linear matrices to the vertices of the boundaries of polyhedral sets and 
  calculating the convex hulls of the polyhedral sets. All such 
  transformations (in the \texttt{Python} language) can be carried out 
  (without loss of computational precision) either using the \texttt{shapely} 
  package (in the case of two-dimensional matrices), or the \texttt{pyhull} 
  or \texttt{qhull} packages (in the case of matrices of dimension greater 
  than $3\times3$). 
\end{remark} 

\begin{remark}\label{rem:other3} 
  Formulas~\eqref{eq:Siter} assume non-degeneracy of matrices $A_{i}$ (or, 
  equivalently, matrices $A_{i}^{T}$). Calculating bodies $M_{n}$ by 
  formula~\eqref{eq:conv-relax} does not require non-degeneracy of the 
  corresponding matrices! Thus, calculating Dranishnikov--Konyagin bodies 
  (and subsequently obtaining the Barabanov body by taking the polar) is less 
  restrictive in computational terms compared to directly calculating the 
  Barabanov norm (body) by the max-relaxation algorithm. 
\end{remark} 

\section{Notes and comments}\label{S:extnorms} 

Let us consider some situations in which extremal norms can be found in some 
reasonable sense ``simply''. 

In 1995, a conjecture was formulated in the paper~\cite{LagWang:LAA95}, which 
has since become known as the ``Lagarias--Wang finiteness conjecture'': 
\begin{quote} 
  \emph{For any set of matrices~$\setA$ the limit~\eqref{eq:GSR} is 
  attained at some finite value $n$. In other words, there exists an $n$ 
  such that in~\eqref{eq:sprad} the equality} 
  \[ 
    \bar{\rho}_{n}({\setA})= \bar{\rho}({\setA})\quad (=\rho({\setA})) 
  \] 
  \emph{holds.} 
\end{quote} 
In 2002, this hypothesis was refuted~\cite{BM:JAMS02}, and a little later 
other versions of its refutation appeared~\cite{BTV:SIAMJMA03, Koz:CDC05:e}. 
Nevertheless, the very formulation of the finiteness hypothesis gave rise to 
numerous studies and had a significant influence on the development 
of the theory of the generalized spectral radius, see, for example, the 
bibliography in~\cite{Koz:IITP13}. 

A similar question can be asked for the joint spectral radius: 
\begin{quote} 
  \emph{Does the equality} 
  \begin{equation}\label{eq:JSReq} 
    (\bar{\rho}({\setA})=)\quad \rho({\setA})=\rho_{n}({\setA}) 
  \end{equation} 
  \emph{always hold for some $n$ in~\eqref{eq:sprad}?} 
\end{quote} 
Here, however, the very formulation of the question requires clarification, 
since the values of $\rho_{n}({\setA})$ depend on the choice of the norm 
$\|\cdot\|$. Therefore, in the case when~$\|\cdot\|$ is taken to be a 
\textit{B}-norm (although a'priori unknown, nevertheless existing!), the 
answer to the question posed is positive: equality~\eqref{eq:JSReq} is 
satisfied already for $n=1$. If the norm $\|\cdot\|$ is arbitrary, then the 
answer to this question is generally negative. 

However, the following question can be asked here: 
\begin{quote} 
  \emph{Let us assume that in some norm $\|\cdot\|$ and for some $n$ 
  equality~\eqref{eq:JSReq} is still satisfied. Can this make it easier to 
  construct an extremal norm for a set of matrices~$\setA$?} 
\end{quote} 
Let us show that the answer to this question is partially positive. 

Let us define by analogy with~\eqref{eq:defrA} for $n\ge1$ and 
$x\in\mathbb{R}^{d}$ the values (seminorms) 
\begin{equation}\label{eq:defn} 
  r_{n}({\setA},x)=\max_{\boldsymbol{\sigma}\in{\{1,\ldots,m\}}^{n}} 
  \|A_{\boldsymbol{\sigma}}x\|= 
  \max_{\boldsymbol{\sigma}\in{\{1,\ldots,m\}}^{n}} \|A_{\sigma_{n}}\cdots 
  A_{\sigma_{2}}A_{\sigma_{1}}x\|. 
\end{equation} 

\begin{lemma}\label{L:main} 
  Let $\varkappa$ be given such that $r_{n}({\setA})\le\varkappa^{n}$ for 
  some $n\ge 1$. Then in the norm 
  \[ 
    \|x\|_{n}:=\max\left\{\|x\|,\frac{1}{\varkappa}r_{1}({\setA},x),\ldots,\frac{1}{\varkappa^{n-1}}r_{n-1}({\setA},x)\right\} 
  \] 
  the inequality 
  \begin{equation}\label{eq:finineq} 
    \max\left\{\|A_{1}x\|_{n},\|A_{2}x\|_{n},\ldots,\|A_{m}x\|_{n}\right\}\le 
    \varkappa\|x\|_{n}, \qquad\forall~x\in\mathbb{R}^{d}, 
  \end{equation} 
  holds. 
\end{lemma} 
\begin{corollary*} 
  If for some $n\ge1$ equality~\eqref{eq:JSReq} holds, then the norm 
  $\|\cdot\|_{n}$ defined by Lemma~\ref{L:main} is extremal for the set of 
  matrices~$\setA$. 
\end{corollary*} 
To prove this corollary, it suffices to note that under its conditions it is 
enough to take $\rho({\setA})$ as $\varkappa$ and then use the assertion of 
Lemma~\ref{L:main}. Therefore, let us proceed to the proof of 
Lemma~\ref{L:main}. 
\begin{proof}[Proof of Lemma~\ref{L:main}] 
  Let us choose some matrix $A_{i}\in\setA$. Then from 
  definition~\eqref{eq:defn} of the quantities $r_{k}(\setA,x)$ we obtain the 
  following inequalities: 
  \begin{align*} 
    \|A_{i}x\|          & \le r_{1}(\setA,x), \\ 
    r_{1}(\setA,A_{i}x) & \le r_{2}(\setA,x), \\ & \cdots              \\ 
    r_{n-1}(\setA,A_{i}x) & \le r_{n}(\setA,x). 
  \end{align*} 
  From here 
  \begin{align*} 
    \|A_{i}x\|_{n}&=                                                                                                                                  
    \max\left\{\|A_{1}x\|,\frac{1}{\varkappa}r_{1}({\setA},A_{1}x),\ldots, 
    \frac{1}{\varkappa^{n-1}}r_{n-1}({\setA},A_{1}x)\right\}
    \\ 
& \le  \max\left\{r_{1}({\setA},x),\frac{1}{\varkappa}r_{2}({\setA},x),\ldots, 
\frac{1}{\varkappa^{n-1}}r_{n}({\setA},x)\right\},
    \\ 
    \intertext{and since by the condition of the lemma $r_{n}(\setA)\le 
    \varkappa^{n}$, then $r_{n}(\setA,x)\le \varkappa^{n} \|x\|$, and 
    therefore}  &\le
    \max\left\{r_{1}({\setA},x),\frac{1}{\varkappa}r_{2}({\setA},x),\ldots, 
    \varkappa\|x\|\right\}\le 
    \\ 
&\le \varkappa\max\left\{\|x\|,\frac{1}{\varkappa}r_{1}({\setA},x),\ldots, 
\frac{1}{\varkappa^{n-1}}r_{n-1}({\setA},x)\right\} =\varkappa \|x\|_{n}. 
  \end{align*} 
  Taking now the maximum over $i$ in the left-hand side of this group of 
  inequalities, we obtain the required statement~\eqref{eq:finineq}. 
\end{proof} 

Lemma~\ref{L:main} naturally raises the following question: 
\begin{quote} 
  \emph{If the extremal norm is known, how (and is it possible) to construct 
  a \textit{B}-norm for the set of matrices~$\setA$ from it?} 
\end{quote} 
One possible answer to this question can be obtained using the following two 
lemmas. 
\begin{lemma}\label{L:ext2bar} 
  Let $\|\cdot\|$ be extremal and $\|\cdot\|_{*}$ be a \textit{B}-norm for a 
  matrix set~$\setA$. Let, in addition, $\rho=\rho(\setA)$, and let the 
  numbers $\alpha>0$ and $\beta>0$ be such that 
  \begin{equation}\label{eq:ext2bar1} 
    \beta\|x\|_{*}\le\|x\|\le\alpha\|x\|_{*}, 
    \qquad\forall~x\in\mathbb{R}^{d}. 
  \end{equation} 

  Then 
  \begin{equation}\label{eq:x0def} 
    \|x\|_{0}=\frac{1}{\rho}\max_{i}\|A_{i}x\| 
  \end{equation} 
  is an extremal norm and the following inequalities hold for it 
  \begin{equation}\label{eq:x0bounds} 
    \beta\|x\|_{*}\le\|x\|_{0}\le\|x\|\le\alpha\|x\|_{*}, 
    \qquad\forall~x\in\mathbb{R}^{d}. 
  \end{equation} 
\end{lemma} 
\begin{proof} 
  First, note that due to the extremality of the norm $\|\cdot\|$, the 
  inequalities 
  \[ 
    \|x\|_{0}= \frac{1}{\rho}\max_{i}\|A_{i}x\| 
    \le\frac{1}{\rho}\max_{i}\{\rho\|x\|\} \le\|x\| 
  \] 
  hold, and then the inequalities 
  \begin{equation}\label{eq:ext2bar1r} 
    \|x\|_{0}\le\|x\|\le\alpha\|x\|_{*}, \qquad\forall~x\in\mathbb{R}^{d}, 
  \end{equation} 
  follow from condition~\eqref{eq:ext2bar1}. 

  Let us now prove that $\beta\|x\|_{*}\le\|x\|_{0}$. To do this, we write 
  out the following chain of inequalities: 
  \begin{align}\nonumber 
    \|x\|_{0} =\frac{1}{\rho}\max_{i}\|A_{i}x\| 
    &\ge\frac{1}{\rho}\max_{i}\beta\|A_{i}x\|_{*}\\
    \label{eq:ext2bar1l} 
    &=\frac{\beta}{\rho}\max_{i}\|A_{i}x\|_{*}= 
    \frac{\beta}{\rho}\rho\|x\|_{*}=\beta\|x\|_{*} 
  \end{align} 
  (here the first equality follows from definition~\eqref{eq:x0def} of the 
  norm $\|\cdot\|_{0}$, the second inequality follows from the left-hand side 
  of condition~\eqref{eq:ext2bar1}, the third equality is obtained by taking 
  the factor $\beta$ out from under the maximum sign, the fourth equality is 
  satisfied by the assumption that $\|\cdot\|_{*}$ is a \textit{B}-norm, the 
  last equality is obvious). From~\eqref{eq:ext2bar1r} 
  and~\eqref{eq:ext2bar1l} inequalities~\eqref{eq:x0bounds} follow. 

  It remains to prove that the norm $\|\cdot\|_{0}$ is extremal. To prove 
  this, we fix an arbitrary index $k\in\{1,\ldots,m\}$ and write out the 
  following chain of inequalities: 
  \[ 
    \|A_{k}x\|_{0}=\frac{1}{\rho}\max_{i}\|A_{i}A_{k}x\| 
    \le\frac{1}{\rho}\max_{i}\rho\|A_{k}x\| =\|A_{k}x\| \le\max_{i}\|A_{i}x\| 
    =\rho\|x\|_{0} 
  \] 
  (here the first equality follows from definition~\eqref{eq:x0def} of the 
  norm $\|\cdot\|_{0}$, the second inequality is true due to the extremality 
  of the norm $\|\cdot\|$, the other relations are obvious). Since in the 
  obtained relations the index $k$ was assumed to be arbitrary, then from 
  them follows the inequality 
  \[ 
    \max_{k}\|A_{k}x\|_{0}\le\rho\|x\|_{0}, 
  \] 
  proving the extremality of the norm $\|\cdot\|_{0}$. This remark completes 
  the proof of Lemma~\ref{L:ext2bar}. 
\end{proof} 

The following lemma shows that the max-relaxation algorithm can be 
significantly simplified if we initially know some extremal norm for the set 
of matrices~$\setA$. 
\begin{lemma}\label{L:extlim} 
  Let~$\setA$ be an irreducible set of matrices and $\|\cdot\|_{0}$ be an 
  extremal norm. Then the sequence of norms 
  \begin{equation}\label{eq:sefnormsn} 
    \|x\|_{n+1}=\frac{1}{\rho}\max_{i}\|A_{i}x\|_{n},\qquad n=0,1,\ldotsb, 
  \end{equation} 
  monotonically decreases and converges to some \textit{B}-norm. 
\end{lemma} 
\begin{proof} 
  Let us take an arbitrary \textit{B}-norm~$\|\cdot\|_{*}$ (according to 
  Barabanov's theorem, it exists due to the assumption of the irreducibility 
  of the set of matrices~$\setA$) and choose such numbers $\alpha$ and 
  $\beta>0$ for which inequalities~\eqref{eq:ext2bar1} from 
  Lemma~\ref{L:ext2bar} are satisfied. Then, by Lemma~\ref{L:ext2bar}, the 
  following relations will hold: \begin{equation}\label{eq:seqnorms} 
  \beta\|x\|_{*}\le\|x\|_{n+1}\le\|x\|_{n}\le\alpha\|x\|_{*},\qquad 
  n=0,1,\ldotsb, 
  \end{equation} 
  which say that the sequence of norms $\{\|x\|_{n}\}$ monotonically 
  decreases (does not increase) pointwise and is bounded from below by the 
  norm $\beta\|x\|_{*}$. Therefore, this sequence converges on each element 
  $x\in\mathbb{R}^{d}$ to some seminorm\footnote{Without additional 
  assumptions, the pointwise limit of norms is only a seminorm, not a norm!} 
  $\|x\|_{\infty}$. And since by~\eqref{eq:seqnorms} this seminorm is bounded 
  from below by the norm $\beta\|x\|_{*}$, it is in fact a norm! 

  It remains to prove that the $\|\cdot\|_{\infty}$ is a Barabanov norm. To 
  do this, it suffices to pass to the limit as ${n\to\infty}$ in 
  equality~\eqref{eq:sefnormsn} defining the sequence of norms 
  $\{\|\cdot\|_{n}\}$. As a result, we obtain: 
  \[ 
    \|x\|_{\infty}=\frac{1}{\rho}\max_{i}\|A_{i}x\|_{\infty}, 
    \quad\text{and}\quad \beta\|x\|_{*}\le\|x\|_{\infty}\le\alpha\|x\|_{*}. 
  \] 
  The lemma is proved. 
\end{proof} 

A result similar to Lemma~\ref{L:extlim} is also true for the construction of 
Dranishnikov–Konyagin bodies. 

\begin{lemma}\label{L:dkplim} 
  Let $\setA=\{A_{1},\ldots,A_{m}\}$ be an irreducible matrix set, 
  $\|\cdot\|$ be some extremal norm for the matrix set~$\setA$, 
  $S=\{x\in\mathbb{R}^d: \|x\|\le 1\}$ be the unit ball in this norm, and 
  $\rho = \rho(\setA)$. Then the sequence of bodies 
  \begin{equation}\label{Eq:Mseq} 
    M_{n+1}=\frac{1}{\rho}\conv\left(\bigcup_{i}A_{i}M_{n}\right),\quad 
    n=0,1,\ldots,\quad\text{where}\quad M_{0}=S, 
  \end{equation} 
  monotonically decreasing converges to some \textit{DK}-body. 
\end{lemma} 
\begin{proof} 
  Since by condition $S$ is a unit ball of the extremal norm $\|\cdot\|$, 
  then  \begin{equation}\label{Eq:Mzero} A_{i}M_{0}=A_{i}S\subseteq\rho 
  S=\rho M_{0},\qquad\forall A_{i}\in\setA. 
  \end{equation} 
  We will prove by induction that in this case for any $n=0,1,\ldots$ the 
  relations 
  \begin{align}\label{Eq:first} 
    A_{i}M_{n}                & \subseteq \rho M_{n}, \\ \label{Eq:second} 
    M_{n+1} & \subseteq M_{n}, 
  \end{align} 
  will be satisfied. 

  Let us first prove the induction statement for $n=0$. 
  Inclusion~\eqref{Eq:first} in this case follows from~\eqref{Eq:Mzero}. Then 
  to prove inclusion~\eqref{Eq:second} it suffices to note that 
  by~\eqref{Eq:Mseq} 
  \[ 
    M_{1}=\frac{1}{\rho}\conv\left(\bigcup_{i}A_{i}M_{0}\right)\subseteq 
    \frac{1}{\rho}\conv(\rho M_{0})=M_{0}. 
  \] 

  Let us perform the induction step: suppose that relations~\eqref{Eq:first} 
  and~\eqref{Eq:second} are true for all $n=0,1,\ldots,k$ and prove that they 
  are true for $n=k+1$. In this case, for each $i=1,2,\ldots,m$ we have: 
  \begin{align*} 
    A_{i}M_{k+1}&\stackrel{(1)}{=}\frac{1}{\rho}A_{i}\conv\left(\bigcup_{j}A_{j}M_{k}\right)= 
    \frac{1}{\rho}\conv\left(\bigcup_{j}A_{i}A_{j}M_{k}\right)\\ 
    &\stackrel{(2)}{\subseteq}\frac{1}{\rho}\conv\left(\rho\bigcup_{j}A_{i}M_{k}\right)= 
    A_{i}M_{k}\subseteq \bigcup_{i}A_{i}M_{k}\\
    &\subseteq\conv\left(\bigcup_{i}A_{i}M_{k}\right)\stackrel{(3)}{=}\rho M_{k+1}. 
  \end{align*} 
  Here equalities (1) and (3) are a consequence of 
  definition~\eqref{Eq:Mseq}, inclusion (2) follows from the inclusion 
  $A_{j}M_{k}\subseteq \rho M_{k}$, which is valid for each $j=1,2\ldots,m$ 
  by the induction hypothesis, and the remaining equalities or inclusions are 
  obvious. 

  It remains to prove inclusion~\eqref{Eq:second} for $n=k+1$. Here it 
  suffices to repeat the calculations carried out earlier for the case $n=0$: 
  \[ 
    M_{k+2}=\frac{1}{\rho}\conv\left(\bigcup_{i}A_{i}M_{k+1}\right)\subseteq 
    \frac{1}{\rho}\conv(\rho M_{k+1})=M_{k+1}. 
  \] 
  The induction step is complete and, therefore, inclusions~\eqref{Eq:first} 
  and~\eqref{Eq:second} are proved. 

  Finally, we define an arbitrary \textit{DK}-body~$M_{*}$ satisfying the 
  relation 
  \[ 
    M_{*}\subseteq S_{0}=M_{0}, 
  \] 
  and show that in this case 
  \begin{equation}\label{Eq:lowbound} 
    M_{*}\subseteq M_{n},\qquad n=0,1,2,\ldots. 
  \end{equation} 
  To prove these inclusions by induction, it suffices to show that 
  \[ 
    M_{*}\subseteq M_{n}\quad\Longrightarrow\quad M_{*}\subseteq M_{n+1}. 
  \] 
  Indeed, by definition and the induction hypothesis 
  \[ 
    M_{n+1}=\frac{1}{\rho}\conv\left(\bigcup_{i}A_{i}M_{n}\right)\supseteq 
    \frac{1}{\rho}\conv\left(\bigcup_{i}A_{i}M_{*}\right)=M_{*}, 
  \] 
  where the last equality holds due to the assumption that $M_{*}$ is a 
  \textit{DK}-body. Thus, relations~\eqref{Eq:lowbound} are proved. 

  It follows from inclusions~\eqref{Eq:second} and~\eqref{Eq:lowbound} that 
  the sequence of bodies $\{M_{n}\}$ is monotonically decreasing (in the 
  sense of the inclusion operation) and is bounded from below (each element 
  contains) a nonzero body $M_{*}$. Then the sequence of bodies $\{M_{n}\}$ 
  converges in the natural sense to some nonzero body~$M$, for which passing 
  to the limit in~\eqref{Eq:Mseq} we obtain the equality 
  \[ 
    M=\frac{1}{\rho}\conv\left(\bigcup_{i}A_{i}M\right)=\boldsymbol{A}M. 
  \] 
  By virtue of Theorem~\ref{T:DK}, the obtained relation says precisely that 
  the limit set~$M$ is a \textit{DK}-body. 

  Lemma~\ref{L:dkplim} is proved. 
\end{proof} 

\begin{remark}\label{rem:alas} 
  Lemmas~\ref{L:ext2bar}--\ref{L:dkplim} are generally of theoretical 
  interest only and are of little use for the practical construction of a 
  \textit{B}-norm from a known extremal norm, since 
  \begin{itemize} 
\item the joint spectral radius $\rho=\rho(\setA)$ a'priori is usually 
    unknown; 
\item Lemmas~\ref{L:extlim} and~\ref{L:dkplim} do not provide any 
    information about the rate of convergence of extremal norms 
    $\{\|\cdot\|_{n}\}$ to the \textit{B}-norm~$\|\cdot\|_{\infty}$. 
  \end{itemize} 
\end{remark} 

However, in one special case, the extremal norm for a set of matrices $\setA$ 
can still be specified explicitly. Let the set $\setA=\{A_{1},\ldots,A_{m}\}$ 
consist of real symmetric $d\times d$ matrices: 
\[ 
  A_{i}^{T}=A_{i},\qquad i=1,2\ldots,m. 
\] 
In this case, in the Euclidean norm 
\[ 
  \|x\|=\sqrt{(x,x)},\qquad x\in\mathbb{R}^{d}, 
\] 
the inequalities 
\[ 
  \|A_{i}\|=\rho(A_{i}),\qquad i=1,2\ldots,m, 
\] 
are satisfied where $\rho(A_{i})$ denotes the spectral radius of the matrix 
$A_{i}$. It follows that the Euclidean norm is extremal for a set of matrices 
$\setA$ whose joint/generalized spectral radius is defined by the equality 
\[ 
  \rho(\setA)=\bar\rho(\setA)=\max_{i}\rho(A_{i}). 
\] 
In this case, the joint/generalized spectral radius is achieved on the matrix 
$A_{i_{*}}$ for which 
\[ 
  \rho(A_{i_{*}})=\rho(\setA). 
\] 

\begin{remark}\label{rem:sym-matsets} 
  A generalization of the set of symmetric matrices is the \emph{symmetric 
  set of matrices} $\setA=\{A_{1},\ldots,A_{m}\}$ introduced 
  in~\cite{PW:LAA08}, which has the property that, together with each matrix, 
  this set also contains a matrix symmetric to it. In this case, the 
  Euclidean norm is also extremal for the set of matrices~$\setA$, and the 
  generalized/joint spectral radius is achieved at one of the matrix 
  products $A^{T}_{i}A_{i}$. 
\end{remark} 

\begin{remark}\label{rem:sym-nonelliptic} 
  At first glance, it is quite unexpected that, unlike the ``nice = smooth = 
  ellipsoidal'' extremal norm for a set of symmetric matrices~$\setA$, the 
  corresponding \textit{B}-norm cannot be found explicitly in general and, as 
  Example~\ref{ex1} shows, has an ``angular'' unit ball. The corresponding 
  \textit{DK}-body is also ``angular''. 
\end{remark} 

\begin{example}\label{ex1} Consider a set of symmetric matrices 
  $\setA=\{A_{1},A_{2}\}$, where 
  \[ 
    A_{1}= 
    \begin{bmatrix}1.1 & 0 \\0 & 0.7 
    \end{bmatrix},\quad 
    A_{2}= 
    \begin{bmatrix}1 & 0.2 \\0.2 & 1 
    \end{bmatrix}. 
  \] 
  For these matrices 
  \[ 
    \rho(A_{1})=1.1,\quad\rho(A_{2})=1.2,\quad\rho(\setA)=1.2, 
  \] 
  where the invariant subspaces of the matrix $A_{1}$ coincide with the 
  coordinate axes, and the invariant subspaces of the matrix $A_{2}$ coincide 
  with the bisectors of the angles between the coordinate axes. The unit ball 
  of the \textit{B}-norm and the \textit{DK}-body for the set of matrices 
  $\setA$ are shown in Fig.~\ref{F:2}. 
\end{example}
  \begin{figure}[htbp!] 
    \centering \subcaptionbox{Unit ball of the Barabanov norm $\|\cdot\|$~--- 
    black solid line,\\ $\{x:\|A_{1}x\|\le \rho\}$~--- red dotted line,\\ 
    $\{x:\|A_{2}x\|\le \rho\}$~--- blue dash-dotted line\label{F:2a}} 
    {\includegraphics*[width=0.48\textwidth]{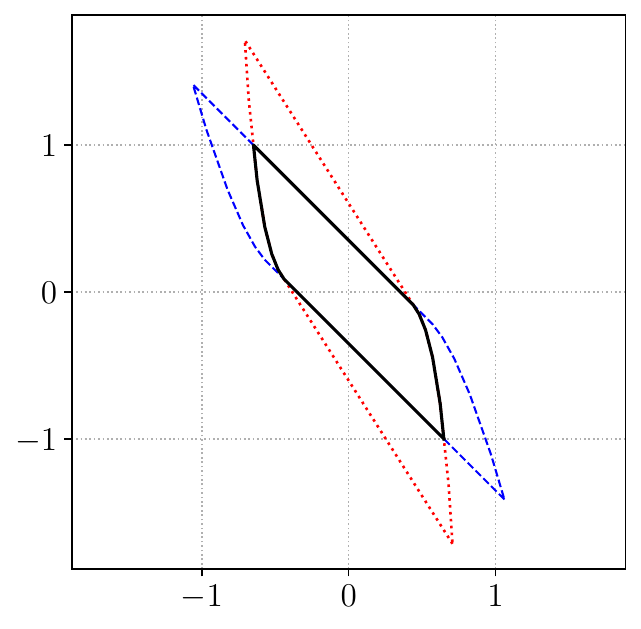}} 
    \hfill\subcaptionbox{Dranishnikov--Konyagin body~--- black solid line,\\ 
    $\rho^{-1}A_{1}M$~--- red dotted line,\\ $\rho^{-1}A_{2}M$~--- blue 
    dashed-dotted line\label{F:2b}} 
    {\includegraphics*[width=0.48\textwidth]{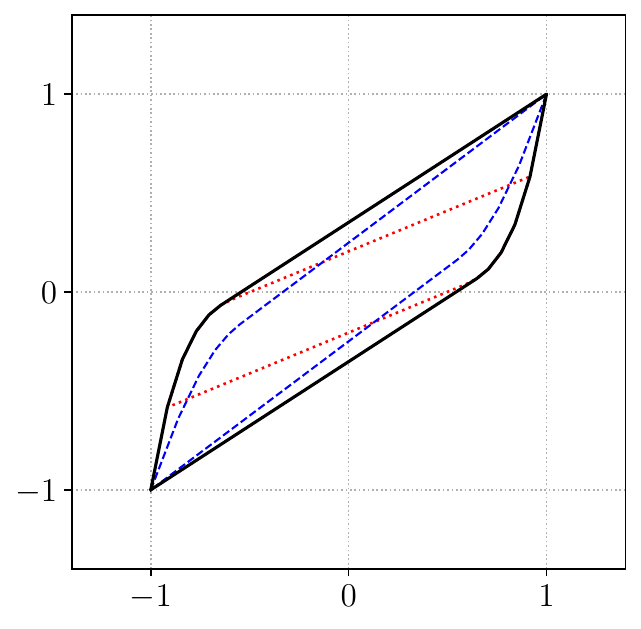}} \caption{The 
    unit ball of the Barabanov norm and the Dranishnikov--Konyagin body for 
    the set of matrices from the set~$\setA$}\label{F:2} 
  \end{figure}

\begin{remark}\label{rem:max-relax} 
  The approximate construction of the \textit{B}-norm in Example~\ref{ex1} 
  was performed using the max-relaxation algorithm~\cite{Koz:DCDSB10, 
  Koz:ArXiv10:1}, examples of software implementation of which are given on 
  the websites \url{https://github.com/kozyakin/barnorm} and 
  \url{https://github.com/kozyakin/spectrum_maximizing_products}. At the same 
  time, the results of Example~\ref{ex1} could have been obtained using a 
  technically somewhat simpler algorithm based on 
  formulas~\eqref{eq:sefnormsn}. 
\end{remark} 

\section{Conclusion}\label{S:conclusion} 
The paper presents a set of techniques that can significantly simplify the 
practical construction of Barabanov norms for a set of matrices. 

It would seem that with such a powerful tool as the 
\href{https://gitlab.com/tommsch/ttoolboxes}{t-toolboxs} extension package 
for MATLAB described in~\cite{Mejstrik:ACMTMS20, MejReif:LAA25}, further 
attempts to develop algorithms for finding the generalized/joint spectral 
radius are mostly pointless. Note, however, that the practical application of 
the \texttt{t-toolboxs} package has a number of limitations: 
\begin{itemize} 
\item firstly, this package is mainly focused on calculating the 
    generalized/joint radius, while constructing extremal norms with its 
    help and, in particular, Barabanov norms requires some additional 
    effort; 
\item secondly, this package is intended for use in the (commercial=paid) 
    MATLAB program and also requires a number of paid MATLAB add-ons for 
    its use; 
\item thirdly, this package is quite large ($\sim15$ Mb). 
\end{itemize} 

The algorithm proposed in this paper is more focused on calculating Barabanov 
norms or Dranishnikov--Konyagin--Protasov bodies/norms. Moreover, as can be 
seen from the listing attached in Appendix~\ref{A:PyCode}, it is implemented 
in a free software environment (Python) and essentially takes up no more than 
150 lines of computer code ($\sim8$~Kb), which makes it applicable for simple 
everyday use in scientific research even by students! 

\appendix 
\section{Program for calculating the Dranishnikov-Konyagin body and the 
Barabanov norm}\label{A:PyCode} The code below is implemented in the 
\texttt{Python} programming language of the \texttt{Python}~3.13.5 
distribution and, together with other examples of calculating the Barabanov 
norm using Dranishnikov--Konyagin bodies, is available for download from the 
site \url{https://github.com/kozyakin/barnorm_via_dkbody}. The modules used 
are \texttt{matplotlib} v3.10.5, \texttt{numpy} v2.3.1, \texttt{shapely} 
v2.1.1. 

\medskip 

\begin{lstlisting}[caption={Python code \texttt{\detokenize{barnorm_v2.py}} to 
compute the Dranishnikov--Konyagin body and the Barabanov norm of a pair of 
matrices}, label=L:code]
# -*- coding: utf-8 -*-
"""Calculation of Barabanov norms through the construction of
Dranishnikov-Konyagin bodies and subsequent calculation of their
polars

To calculate the polygonal approximation of the unit sphere in the
Barabanov norm, this program first calculates the polygonal
approximation of the corresponding Dranishnikov-Konyagin body, and
then, by taking its polar, calculates the polygonal approximation
of the unit sphere of the Barabanov norm. This approach, in contrast
to the direct implementation of the max-relaxation algorithm,
eliminates the need to calculate inverse matrices and, thus, expands
the capabilities of the max-relaxation algorithm.

Version 2.0
Released on  2025-07-08 10:55:28 +0300
Last updated 2025-08-02 14:42:38 +0300

Version 1.0 was based on direct calculation of the Barabanov norm,
without intermediate calculation of the Dranishnikov--Konyagin norm
(body)
Released on  2019-09-21 12:37:46 +0300

@author: Victor Kozyakin
"""

import math
import os
import platform
import time
from importlib.metadata import version
from pathlib import Path

import numpy as np
import shapely
import shapely.affinity
from matplotlib import pyplot
from matplotlib.ticker import MultipleLocator
from shapely.geometry import LineString, MultiPoint

# Initial data

TOL = 0.0000001
BB_SCALE = 1.4

ALPHA = 0.576
BETA = 0.8
AAA = 0.9
BBB = 1.1
CCC = 1.0
DDD = 0.9

A1 = ALPHA * np.asarray([[AAA, BBB], [0.0, 1.0]])
A2 = BETA * np.asarray([[1.0, 0.0], [CCC, DDD]])

# Functions


def polygonal_norm(_x, _y, _h):
    """Calculate the norm specified by a polygonal unit ball.

    Args:
        _x (real): x-coordinate of vector
        _y (real): y-coordinate of vector
        _h (MultiPoint): polygonal norm unit ball

    Returns:
        real: vector's norm
    """
    _hb = _h.bounds
    _scale = 0.5 * math.sqrt(
        ((_hb[2] - _hb[0]) ** 2 + (_hb[3] - _hb[1]) ** 2)
        / (_x**2 + _y**2)
    )
    _ll = LineString([(0, 0), (_scale * _x, _scale * _y)])
    _p_int = _ll.intersection(_h).coords
    return math.sqrt(
        (_x**2 + _y**2) / (_p_int[1][0] ** 2 + _p_int[1][1] ** 2)
    )


def min_max_norms_quotent(_g, _h):
    """Calculate the min/max of the quotient g-norm/h-norm.

    Args:
        _g (MultiPoint): polygonal norm unit ball
        _h (MultiPoint): polygonal norm unit ball

    Returns:
        2x0-array: mimimum and maximum of g-norm/h-norm
    """
    _pg = _g.boundary.coords
    _dimg = len(_pg) - 1
    _sg = [
        1 / polygonal_norm(_pg[i][0], _pg[i][1], _h)
        for i in range(_dimg)
    ]
    _ph = _h.boundary.coords
    _dimh = len(_ph) - 1
    _sh = [
        polygonal_norm(_ph[i][0], _ph[i][1], _g) for i in range(_dimh)
    ]
    _sgh = _sg + _sh
    return (min(_sgh), max(_sgh))


def polar_polygon(_inpol):
    """Calculation of polar (dual) polygon for input polygon

    Args:
        _inpol: dx2 ndarray or shapely MultiPoint/Poligon structure
        of coordinates of input polygon vertices

    Returns:
        dx2 ndarray of coordinates of polar (dual) polygon vertices
    """
    _ort_mat = np.asarray([[0, -1], [1, 0]])
    if type(_inpol).__name__ == "MultiPoint":
        _inpol = _inpol.convex_hull
    if type(_inpol).__name__ == "Polygon":
        _inpol = np.asarray(_inpol.boundary.coords)
    _inpol = np.concatenate((_inpol, -_inpol), axis=0)
    _inpol = MultiPoint(_inpol)
    _inpol = _inpol.convex_hull
    _inpol = np.asarray(_inpol.boundary.coords)

    _dual = np.zeros((len(_inpol), 2))

    for _i in range(len(_inpol) - 1):
        _dual[_i] = _ort_mat @ (
            _inpol[_i + 1] - _inpol[_i]
        )  # see numpy @
        _dual[_i] = _dual[_i] / (_dual[_i] @ _inpol[_i])
    _dual[-1] = _dual[0]
    return _dual


# Computation initialization

t_tick = time.time()

pgn = np.asarray([[1, -1], [1, 1]])
pgn = np.concatenate((pgn, -pgn), axis=0)
scale = 1 / np.max(pgn)
pgn = MultiPoint(pgn)

hull = pgn.convex_hull
hull = shapely.affinity.scale(hull, xfact=scale, yfact=scale)

print("\n  #    rho_min     rho     rho_max  residual  Num_edges\n")

# Iterative computing Dranishnikov-Konyagin + Barabanov bodies

NITER = 0
while True:
    pgn = np.asarray(hull.boundary.coords)
    pgn_A1 = MultiPoint(pgn @ A1.T)
    hull_A1 = pgn_A1.convex_hull

    pgn_A2 = MultiPoint(pgn @ A2.T)
    hull_A2 = pgn_A2.convex_hull

    hull_A12 = pgn_A1.union(pgn_A2).convex_hull
    pgn_A12 = MultiPoint(hull_A12.boundary.coords)

    rho_minmax = min_max_norms_quotent(hull, hull_A12)
    rho_max = rho_minmax[1]
    rho_min = rho_minmax[0]
    rho = (rho_max + rho_min) / 2

    residual = rho_max / rho_min - 1.0

    hull = hull.union(
        shapely.affinity.scale(hull_A12, xfact=1 / rho, yfact=1 / rho)
    )
    hull = hull.convex_hull

    NITER += 1
    print(
        f"{NITER:3.0f}.",
        f"{rho_min: .6f}",
        f"{rho: .6f}",
        f"{rho_max: .6f}",
        f"{residual: .6f}",
        "    ",
        len(hull.boundary.coords) - 1,
    )
    scale = 1 / np.max(hull.boundary.coords)
    hull = shapely.affinity.scale(hull, xfact=scale, yfact=scale)

    if (residual) < TOL:
        break

barnorm_sphere = polar_polygon(hull)

comptime = time.time() - t_tick

# Plotting Dranishnikov-Konyagin + Barabanov bodies
#
# Calculations of boundary points (p10 and p20) of polygons h10 and
# h20 differ in the case if this polygon degenerates into a straight
# line segment (Type = LineString) and in the case if this polygon
# is non-degenerate - has internal points (Type = Polygon)

scale = max(np.max(barnorm_sphere), -np.min(barnorm_sphere))
barnorm_sphere = barnorm_sphere / scale

hull_A1_scaled = shapely.affinity.scale(
    hull_A1, xfact=1 / rho, yfact=1 / rho
)
if type(hull_A1_scaled).__name__ == "LineString":
    pgn_A1_scaled = np.asarray(hull_A1_scaled.coords)
else:
    pgn_A1_scaled = np.asarray(hull_A1_scaled.boundary.coords)

hull_A2_scaled = shapely.affinity.scale(
    hull_A2, xfact=1 / rho, yfact=1 / rho
)
if type(hull_A2_scaled).__name__ == "LineString":
    pgn_A2_scaled = np.asarray(hull_A2_scaled.coords)
else:
    pgn_A2_scaled = np.asarray(hull_A2_scaled.boundary.coords)

bb = 1.2 * max(
    np.max(barnorm_sphere),
    np.max(pgn),
    np.max(pgn_A1_scaled),
    np.max(pgn_A2_scaled),
)

# Plotting Dranishnikov-Konyagin body

fig1 = pyplot.figure(num="Dranishnikov-Konyagin + Barabanov bodies")
ax1 = fig1.add_subplot(111)
ax1.set_xlim(-BB_SCALE * bb, BB_SCALE * bb)
ax1.set_ylim(-BB_SCALE * bb, BB_SCALE * bb)
ax1.set_aspect(1)
ax1.grid(True, linestyle=":")
ax1.xaxis.set_major_locator(MultipleLocator(1))
ax1.yaxis.set_major_locator(MultipleLocator(1))

ax1.plot(
    pgn_A1_scaled[:, 0],
    pgn_A1_scaled[:, 1],
    ":",
    color="red",
    linewidth=1.25,
)
ax1.plot(
    pgn_A2_scaled[:, 0],
    pgn_A2_scaled[:, 1],
    "--",
    color="blue",
    linewidth=1,
)
ax1.plot(pgn[:, 0], pgn[:, 1], "-", color="black")
ax1.plot(
    barnorm_sphere[:, 0], barnorm_sphere[:, 1], "-", color="green"
)

pyplot.show()

# Saving plot to pdf-files

current_filename = os.path.basename(__file__)
filename_without_extension = Path(current_filename).stem

# fig1.savefig(
#     filename_without_extension + ".pdf", bbox_inches="tight")

print(
    "\nModules used:  Python " + platform.python_version(),
    "matplotlib " + version("matplotlib"),
    "numpy " + version("numpy"),
    "shapely " + version("shapely"),
    sep=", ",
)

print("Computation time: ", f"{round(comptime, 6):6.2f} sec.")
\end{lstlisting} 


\begin{thebibliography}{10}
	\def\mrref#1{\href{https://www.ams.org/mathscinet-getitem?mr=#1}{MR~#1}}
	\def\zblref#1{\href{https://zbmath.org/?q=an:#1}{Zbl~#1}}
	\expandafter\ifx\csname url\endcsname\relax
	\def\url#1{\texttt{#1}}\fi
	\expandafter\ifx\csname urlprefix\endcsname\relax\def\urlprefix{URL }\fi
	\expandafter\ifx\csname href\endcsname\relax
	\def\href#1#2{#2} \def\path#1{#1}\fi
	\providecommand{\bbljan}[0]{January} \providecommand{\bblfeb}[0]{February}
	\providecommand{\bblmar}[0]{March} \providecommand{\bblapr}[0]{April}
	\providecommand{\bblmay}[0]{May} \providecommand{\bbljun}[0]{June}
	\providecommand{\bbljul}[0]{July} \providecommand{\bblaug}[0]{August}
	\providecommand{\bblsep}[0]{September} 
	\providecommand{\bbloct}[0]{October}
	\providecommand{\bblnov}[0]{November} 
	\providecommand{\bbldec}[0]{December}
	\providecommand{\bbljan}[0]{January} \providecommand{\bblfeb}[0]{February}
	\providecommand{\bblmar}[0]{March} \providecommand{\bblapr}[0]{April}
	\providecommand{\bblmay}[0]{May} \providecommand{\bbljun}[0]{June}
	\providecommand{\bbljul}[0]{July} \providecommand{\bblaug}[0]{August}
	\providecommand{\bblsep}[0]{September} 
	\providecommand{\bbloct}[0]{October}
	\providecommand{\bblnov}[0]{November} 
	\providecommand{\bbldec}[0]{December}
	
	\bibitem{Bar:AIT88-2:e}
	N.~E. Barabanov, \emph{{L}yapunov indicator of discrete inclusions. {I}},
	Autom. Remote Control \textbf{49} (1988), no.~2, 152--157. \mrref{940263}.
	\zblref{0665.93043}.
	
	\bibitem{Bar:AIT88-3:e}
	N.~E. Barabanov, \emph{The {L}yapunov indicator of discrete inclusions. 
	{II}},
	Autom. Remote Control \textbf{49} (1988), no.~3, 283--287. \mrref{943889}.
	\zblref{0665.93044}.
	
	\bibitem{Bar:AIT88-5:e}
	N.~E. Barabanov, \emph{The {L}yapunov indicator of discrete inclusions. 
	{III}},
	Autom. Remote Control \textbf{49} (1988), no.~5, 558--565. \mrref{952665}.
	\zblref{0665.93045}.
	
	\bibitem{Bar:ACC95}
	N.~E. Barabanov, \emph{Stability of inclusions of linear type}, American
	Control Conference, Proceedings of the 1995, vol.~5, June 1995,
	pp.~3366--3370.
	\newblock \href{https://doi.org/10.1109/ACC.1995.532231}
	{\path{doi:10.1109/ACC.1995.532231}}.
	
	\bibitem{BerWang:LAA92}
	M.~A. Berger and Y.~Wang, \emph{Bounded semigroups of matrices}, Linear 
	Algebra
	Appl. \textbf{166} (1992), 21--27.
	\newblock \href{https://doi.org/10.1016/0024-3795(92)90267-E}
	{\path{doi:10.1016/0024-3795(92)90267-E}}. \mrref{1152485}.
	\zblref{0818.15006}.
	
	\bibitem{BTV:SIAMJMA03}
	V.~D. Blondel, J.~Theys, and A.~A. Vladimirov, \emph{An elementary
		counterexample to the finiteness conjecture}, SIAM J. Matrix Anal. 
		Appl.
	\textbf{24} (2003), no.~4, 963--970 (electronic).
	\newblock \href{https://doi.org/10.1137/S0895479801397846}
	{\path{doi:10.1137/S0895479801397846}}. \mrref{2003315}. 
	\zblref{1043.15007}.
	
	\bibitem{BM:JAMS02}
	T.~Bousch and J.~Mairesse, \emph{Asymptotic height optimization for 
	topical
		{IFS}, {T}etris heaps, and the finiteness conjecture}, J. Amer. Math. 
		Soc.
	\textbf{15} (2002), no.~1, 77--111 (electronic).
	\newblock \href{https://doi.org/10.1090/S0894-0347-01-00378-2}
	{\path{doi:10.1090/S0894-0347-01-00378-2}}. \mrref{1862798}.
	\zblref{1057.49007}.
	
	\bibitem{DaubLag:LAA92}
	I.~Daubechies and J.~C. Lagarias, \emph{Sets of matrices all infinite 
	products
		of which converge}, Linear Algebra Appl. \textbf{161} (1992), 
		227--263.
	\newblock \href{https://doi.org/10.1016/0024-3795(92)90012-Y}
	{\path{doi:10.1016/0024-3795(92)90012-Y}}. \mrref{1142737}.
	\zblref{0746.15015}.
	
	\bibitem{GugProt:FCM13}
	N.~Guglielmi and V.~Protasov, \emph{Exact computation of joint spectral
		characteristics of linear operators}, Found. Comput. Math. \textbf{13}
	(2013), no.~1, 37--97.
	\newblock \href{https://doi.org/10.1007/s10208-012-9121-0}
	{\path{doi:10.1007/s10208-012-9121-0}},
	\href{https://arxiv.org/abs/1106.3755} {\path{arXiv:1106.3755}}.
	\mrref{3009529}. \zblref{06153962}.
	
	\bibitem{GugZen:CDC05}
	N.~Guglielmi and M.~Zennaro, \emph{Polytope norms and related algorithms 
	for
		the computation of the joint spectral radius}, Proceedings of the 
		44th {IEEE}
	Conference on Decision and Control and European Control Conference 2005,
	Seville, Spain, December 12--15, 2005, pp.~3007--3012.
	\newblock \href{https://doi.org/10.1109/CDC.2005.1582622}
	{\path{doi:10.1109/CDC.2005.1582622}}.
	
	\bibitem{HJ2:e}
	R.~A. Horn and C.~R. Johnson, \emph{Matrix analysis}, second ed., 
	Cambridge
	University Press, Cambridge, 2013. \mrref{2978290}. \zblref{1267.15001}.
	
	\bibitem{Jungers:09}
	R.~Jungers, \emph{The joint spectral radius}, Lecture Notes in Control and
	Information Sciences, vol. 385, Springer-Verlag, Berlin, 2009.
	\newblock {T}heory and applications.
	\newblock \href{https://doi.org/10.1007/978-3-540-95980-9}
	{\path{doi:10.1007/978-3-540-95980-9}}. \mrref{2507938}.
	
	\bibitem{Koz:CDC05:e}
	V.~Kozyakin, \emph{A dynamical systems construction of a counterexample 
	to the
		finiteness conjecture}, Proceedings of the 44th IEEE Conference on 
		Decision
	and Control, 2005 and 2005 European Control Conference. CDC-ECC'05., 2005,
	pp.~2338--2343.
	\newblock \href{https://doi.org/10.1109/CDC.2005.1582511}
	{\path{doi:10.1109/CDC.2005.1582511}}.
	
	\bibitem{Koz:AiT90:6:e}
	V.~S. Kozyakin, \emph{Algebraic unsolvability of problem of absolute 
	stability
		of desynchronized systems}, Autom. Remote Control \textbf{51} (1990), 
		no.~6,
	754--759. \mrref{1071607}. \zblref{0737.93056}.
	
	\bibitem{Koz:AiT03:9:e}
	V.~S. Kozyakin, \emph{Indefinability in o-minimal structures of finite 
	sets of
		matrices whose infinite products converge and are bounded or 
		unbounded},
	Autom. Remote Control \textbf{64} (2003), no.~9, 1386--1400.
	\newblock \href{https://doi.org/10.1023/A:1026091717271}
	{\path{doi:10.1023/A:1026091717271}}. \mrref{2090805}. 
	\zblref{1078.93017}.
	
	\bibitem{Koz:INFOPROC05:e}
	V.~S. Kozyakin, \emph{Rotation numbers of discontinuous 
	orientation-preserving
		circle maps revisited}, Information Processes \textbf{5} (2005), 
		no.~4,
	301--335.
	\newblock \urlprefix\url{http://www.jip.ru/2005/283-300.pdf}.
	
	\bibitem{Koz:INFOPROC06:e}
	V.~S. Kozyakin, \emph{Structure of extremal trajectories of discrete 
	linear
		systems and the finiteness conjecture}, Autom. Remote Control 
		\textbf{68}
	(2007), no.~1, 174--209.
	\newblock \href{https://doi.org/10.1134/S0005117906040171}
	{\path{doi:10.1134/S0005117906040171}}. \zblref{1195.93082}.
	
	\bibitem{Koz:DAN09:e}
	V.~S. Kozyakin, \emph{On the computational aspects of the theory of joint
		spectral radius}, Dokl. Math. \textbf{80} (2009), no.~1, 487--491.
	\newblock \href{https://doi.org/10.1134/S1064562409040097}
	{\path{doi:10.1134/S1064562409040097}}. \mrref{2573049}. 
	\zblref{1190.93035}.
	
	\bibitem{Koz:DCDSB10}
	V.~Kozyakin, \emph{Iterative building of {B}arabanov norms and 
	computation of
		the joint spectral radius for matrix sets}, Discrete Contin. Dyn. 
		Syst. Ser.
	B \textbf{14} (2010), no.~1, 143--158.
	\newblock \href{https://doi.org/10.3934/dcdsb.2010.14.143}
	{\path{doi:10.3934/dcdsb.2010.14.143}},
	\href{https://arxiv.org/abs/0810.2154} {\path{arXiv:0810.2154}}.
	\mrref{2644257}. \zblref{1201.65067}.
	
	\bibitem{Koz:ArXiv10:1}
	V.~Kozyakin, \emph{Max-{R}elaxation iteration procedure for building of
		{B}arabanov norms: Convergence and examples}, ArXiv.org e-Print 
		archive,
	February 2010.
	\newblock \href{https://arxiv.org/abs/1002.3251} {\path{arXiv:1002.3251}}.
	
	\bibitem{Koz:IITP13}
	V.~Kozyakin, \emph{An annotated bibliography on the convergence of matrix
		products and the theory of joint/generalized spectral radius}, 
		Preprint,
	Institute for Information Transmission Problems, Moscow, December 2013.
	\newblock \href{https://doi.org/10.13140/RG.2.1.4257.5040/1}
	{\path{doi:10.13140/RG.2.1.4257.5040/1}}.
	
	\bibitem{LagWang:LAA95}
	J.~C. Lagarias and Y.~Wang, \emph{The finiteness conjecture for the 
	generalized
		spectral radius of a set of matrices}, Linear Algebra Appl. 
		\textbf{214}
	(1995), 17--42.
	\newblock \href{https://doi.org/10.1016/0024-3795(93)00052-2}
	{\path{doi:10.1016/0024-3795(93)00052-2}}. \mrref{1311628}.
	\zblref{0818.15007}.
	
	\bibitem{Mejstrik:ACMTMS20}
	T.~Mejstrik, \emph{Algorithm 1011: {I}mproved invariant polytope 
	algorithm and
		applications}, ACM Trans. Math. Software \textbf{46} (2020), no.~3, 
		Art. 29,
	26.
	\newblock \href{https://doi.org/10.1145/3408891} 
	{\path{doi:10.1145/3408891}},
	\href{https://arxiv.org/abs/1812.03080} {\path{arXiv:1812.03080}}.
	\mrref{4161245}. \zblref{1484.65090}.
	
	\bibitem{MejReif:LAA25}
	T.~Mejstrik and U.~Reif, \emph{A hybrid approach to joint spectral radius
		computation}, Linear Algebra Appl. (2025).
	\newblock In Press, Corrected Proof.
	\newblock \href{https://doi.org/10.1016/j.laa.2025.06.024}
	{\path{doi:10.1016/j.laa.2025.06.024}},
	\href{https://arxiv.org/abs/2308.05244} {\path{arXiv:2308.05244}}.
	
	\bibitem{PW:LAA08}
	E.~Plischke and F.~Wirth, \emph{Duality results for the joint spectral 
	radius
		and transient behavior}, Linear Algebra Appl. \textbf{428} (2008), 
		no.~10,
	2368--2384.
	\newblock \href{https://doi.org/10.1016/j.laa.2007.12.009}
	{\path{doi:10.1016/j.laa.2007.12.009}}. \mrref{2408033}. 
	\zblref{05268361}.
	
	\bibitem{PWB:CDC05}
	E.~Plischke, F.~Wirth, and N.~Barabanov, \emph{Duality results for the 
	joint
		spectral radius and transient behavior}, Proceedings of the 44th 
		{IEEE}
	Conference on Decision and Control and European Control Conference 2005,
	Seville, Spain, December 12--15, 2005, pp.~2344--2349.
	\newblock \href{https://doi.org/10.1109/CDC.2005.1582512}
	{\path{doi:10.1109/CDC.2005.1582512}}.
	
	\bibitem{Prot:FPM96:e}
	V.~{\relax{}Yu}. Protasov, \emph{The joint spectral radius and invariant 
	sets
		of linear operators}, Fundam. Prikl. Mat. \textbf{2} (1996), no.~1, 
		205--231.
	\newblock In Russian. \mrref{1789006}. \zblref{0899.47002}.
	
	\bibitem{RobRob:e}
	A.~P. Robertson and W.~J. Robertson, \emph{Topological vector spaces},
	Cambridge Tracts in Mathematics and Mathematical Physics, No. 53, 
	Cambridge
	University Press, New York, 1964. \mrref{0162118}.
	
	\bibitem{RotaStr:IM60}
	G.-C. Rota and G.~Strang, \emph{A note on the joint spectral radius}, 
	Indag.
	Math. \textbf{63} (1960), 379--381.
	\newblock Original journal title: Nederl. Akad. Wetensch. Proc. Ser. A 
	{\bf
		63}.
	\newblock \href{https://doi.org/10.1016/S1385-7258(60)50046-1}
	{\path{doi:10.1016/S1385-7258(60)50046-1}}. \mrref{0147922}.
	\zblref{0095.09701}.
	
	\bibitem{TB:MCSS97:1}
	J.~N. Tsitsiklis and V.~D. Blondel, \emph{The {L}yapunov exponent and 
	joint
		spectral radius of pairs of matrices are hard --- when not impossible 
		--- to
		compute and to approximate}, Math. Control Signals Systems \textbf{10}
	(1997), no.~1, 31--40.
	\newblock \href{https://doi.org/10.1007/BF01219774}
	{\path{doi:10.1007/BF01219774}}. \mrref{1462278}. \zblref{0888.65044}.
	
	\bibitem{VHJ:ACM14}
	G.~Vankeerberghen, J.~Hendrickx, and R.~M. Jungers, \emph{{JSR}: {A} 
	toolbox to
		compute the joint spectral radius}, Proceedings of the 17th 
		International
	Conference on Hybrid Systems: Computation and Control (New York, NY, USA),
	HSCC'14, ACM, 2014, pp.~151--156.
	\newblock \href{https://doi.org/10.1145/2562059.2562124}
	{\path{doi:10.1145/2562059.2562124}}. \mrref{3388353}. 
	\zblref{1364.65099}.
	\KeyWords{joint spectral radius; {M}atlab; stability; switched systems}
	
	\bibitem{Wirth:LAA02}
	F.~Wirth, \emph{The generalized spectral radius and extremal norms}, 
	Linear
	Algebra Appl. \textbf{342} (2002), 17--40.
	\newblock \href{https://doi.org/10.1016/S0024-3795(01)00446-3}
	{\path{doi:10.1016/S0024-3795(01)00446-3}}. \mrref{1873424}.
	\zblref{0996.15020}.
	
\end{thebibliography}

\providecommand{\KeyWords}[1]{#1}

\end{document}